\documentclass[a4paper,11pt]{amsart}

\usepackage[latin1]{inputenc}
\usepackage[english]{babel}
\usepackage{amsmath}
\usepackage{amsthm}
\usepackage{amssymb}
\usepackage{amscd}
\usepackage[all]{xy}
\usepackage{hyperref}
\usepackage{geometry}
\usepackage{pb-diagram}
\usepackage{color}

\title[On some notions of good reduction]{On some notions of good reduction for endomorphisms of the projective line}
\author{Jung Kyu Canci, Giulio Peruginelli, Dajano Tossici}

\newtheorem{Teo}{Theorem}[section]
\newtheorem{Prop}[Teo]{Proposition}
\newtheorem{Lemma}[Teo]{Lemma}
\newtheorem{Cor}[Teo]{Corollary}
\theoremstyle{definition}
\newtheorem{Def}[Teo]{Definition}
\newtheorem{Esempio}[Teo]{Example}
\newtheorem{Oss}[Teo]{Remark}
\newcommand{\Dim}{{\bfseries Proof : }}
\newcommand{\Q}{\mathbb{Q}}

\newcommand{\N}{\mathbb{N}}
\newcommand{\Z}{\mathbb{Z}}

\newcommand{\SR}{\mathbb{P}^1}
\newcommand{\PL}{\mathbb{P}^1}

\begin{document}

%
\address{Departement Math. Universit\"at Basel\\Rheinsprung 21\\CH-4051 Basel\\ Switzerland.}
\email[Canci]{jkcanci@yahoo.it}
\address{Institut f\"ur Analysis und Comput. Number Theory\\Technische Universit\"at\\Steyrergasse 30\\ A-8010 Graz\\Austria.}
\email[Peruginelli]{peruginelli@math.tugraz.at}
\address{Scuola Normale Superiore\\Piazza dei Cavalieri 7\\
56126 Pisa\\ Italy.}
\email[Tossici]{dajano.tossici@sns.it}

\bigskip

\begin{abstract}\noindent Let $\Phi$ be an endomorphism of $\SR(\overline{\Q})$, the projective line over the algebraic closure of $\Q$, of degree $\geq2$ defined over a
number field $K$. Let $v$ be a non-archimedean valuation of $K$. We say that $\Phi$
has critically good reduction at $v$ if any pair of distinct ramification
points of $\Phi$ do not collide under reduction modulo $v$ and the same holds for
any pair of branch points. We say that $\Phi$ has simple good reduction at  $v$ if the map $\Phi_v$, the reduction of $\Phi$ modulo $v$,
has the same degree of $\Phi$.
We prove that if $\Phi$ has critically good reduction at $v$ and the reduction map $\Phi_v$ is separable,
then $\Phi$ has simple good reduction at $v$.
\end{abstract}
\maketitle
\section{Introduction}

Throughout this paper $K$ will be a number field and $\overline{\Q}$ the algebraic closure of $\Q$. More generally, for any arbitrary field $\Omega$, the symbol  $\overline{\Omega}$ will always denote the algebraic closure of $\Omega$.

In a recent paper  Szpiro and Tucker (\cite{SzT})  use a particular notion of good reduction
to prove a finiteness result for equivalence classes of endomorphisms of $\SR(\overline{\Q})$,
which we will indicate simply with $\SR$ in the sequel.
This result implies the Shafarevich-Faltings finiteness theorem for isomorphism classes of elliptic curves.

Let us recall the definition of good reduction used by Szpiro and Tucker, but before we fix some notation.
Let  $O_K$ be the ring of integers of $K$. For a fixed finite place $v$ of $K$ let  $O_v$ be the valuation ring and let $k(v)$ be the residue field. We will not distinguish between the place $v$ and the associated valuation.  Let $S$ be a fixed finite set of places of $K$ containing all the archimedean ones. We denote by $O_S$ the set of $S$-integers, namely
$$ O_S\doteqdot\{x\in K\mid |x|_v\leq 1 \ \text{for all}\ v\notin S\}.$$

Let $\Phi$ be an endomorphism of $\SR$ defined over $K$. We denote with $\mathcal{R}_{\Phi}$ the set of ramification points defined
over $\overline{\Q}$ of the map $\Phi$. Given a valuation $v$ of
$\overline{\Q}$ and a subset $E\subset\SR(\overline{\Q})$ we denote
with $(E)_v$ the subset of $\SR(\overline{k(v)})$, whose elements are
the reduction modulo $v$ of the elements of $E$.

Now we are ready to give the definition of good reduction used by Szpiro and Tucker in
\cite{SzT}:

\begin{Def}
Suppose that $v$ has been extended to $\overline{\Q}$. Let $\Phi$ be an endomorphism of $\SR$ of degree $\geq 2$ defined over $K$. We say that
$\Phi$ has \emph{critically good reduction} (in the sequel C.G.R.) at $v$  if

1) $\#\mathcal{R}_{\Phi}=\#(\mathcal{R}_{\Phi})_v$,\smallskip

2)$\#\Phi(\mathcal{R}_{\Phi})=\#(\Phi(\mathcal{R}_{\Phi}))_v$.
\end{Def}


\noindent As the authors of \cite{SzT} note,
this definition does not depend on the extension of $v$ to
$\overline{\Q}$.

We denote by  ${\rm PGL}(2,O_S)$ the group obtained as quotient of ${\rm GL}(2,O_S)$ modulo scalar matrices.
The group ${\rm PGL}(2,O_S)$ acts on $\SR$: to any element $\Gamma\in{\rm PGL}(2,O_S)$ we associate in a canonical
way an automorphism $\gamma$ of $\SR$. In \cite{SzT} the following equivalence relation on the set of endomorphisms of $\SR$ is used:
two such morphisms $\Psi$ and $\Phi$ are equivalent if there exist automorphisms $\gamma,\sigma$ associated to two elements
in $\in {\rm PGL}(2, O_S)$ such that
$$ \Psi=\gamma\circ\Phi\circ\sigma.$$

With the above notations and definitions, the main result in \cite{SzT} says that for each fixed positive integer
$n$ and fixed number field $K$ there are finitely many equivalence classes
of rational maps $\Phi:\SR\to\SR$ defined over $K$ of degree $n$
that ramify at three or more points and have C.G.R. at any valuation $v$ outside a
prescribed set $S$ of places of $K$, which includes all the archimedean
ones.


There is another definition of good reduction in the context of endomorphisms of $\SR$.
Let $\Phi:\SR\to\SR$ be a rational map defined over $K$, of the
form
$$\Phi([X:Y])=[F(X,Y):G(X,Y)]$$ where $F,G\in K[X,Y]$ are
coprime homogeneous polynomials of the same degree. We may suppose
that the coefficients of $F$ and $G$ are in $O_v$ and factoring out
any common factor we may suppose that at least one of them is a
$v$-unit. If that happens, we will say that $\Phi$ is in $v$-reduced form. So we may give the following definition:
\begin{Def}
Let $\Phi:\SR\to\SR$ be a rational map defined over $K$ and $v$ a finite place of $K$.
Suppose that $\Phi([X:Y])=[F(X,Y):G(X,Y)]$ is written in $v$-reduced form as above.
The reduced map $\Phi_v:\SR_{k(v)}\to\SR_{k(v)}$ is defined by $[F_v(X,Y):G_v(X,Y)]$, where $F_v$ and $G_v$ are
the polynomials obtained from $F$ and $G$  by reducing their coefficients modulo $v$.
\end{Def}
The second notion of good reduction we are going to consider is the following (first appeared in \cite{MoS}, but see also \cite{Silverman}):
\begin{Def}
A rational map $\Phi\colon\SR\to\SR$, defined over $K$, has \emph{simple good reduction} (in the sequel S.G.R) at a place $v$ if  $\deg\Phi=\deg{\Phi}_v$.
\end{Def}

Roughly speaking, in the above notation, $\Phi$ has S.G.R. at $v$ if $F_v$ and $G_v$  have no common factors. Alternatively, from a schematic point of view, the above definition means the following: if we consider $\Phi$ as a scheme morphism $\Phi:\mathbb{P}^1_K\to \mathbb{P}^1_K$ then $\Phi$ has S.G.R. at $v$ if there exists a morphism
$\Phi_{O_v}:\PL_{O_v}\to \PL_{O_v}$ which extends $\Phi$, i.e. the following diagram
\begin{center}
$
\begin{diagram}
\node{\ \mathbb{P}^{1}_{K}}\arrow{e,t}{\Phi} \arrow{s,l}{}
\node{\ \mathbb{P}^{1}_{K}}\arrow{s,r}{}\\
\node{\ \mathbb{P}^{1}_{O_v}}\arrow{e,t}{\Phi_{O_v}}\node{\ \mathbb{P}^{1}_{O_v}}
\end{diagram}
$
\end{center} is commutative, where the vertical maps are the natural open immersions.

The definition of simple good reduction is, perhaps, more natural than the definition of
critically good reduction. But note that a rational map on $\mathbb{P}^1(K)$ associated
to a polynomial in $K[z]$ has simple good reduction outside $S$ if and only the coefficients of
the polynomial are $S$-integers and its leading coefficient is an $S$-unit.
Therefore for sufficiently large $n$ the main theorem of \cite{SzT} would be
false if one considered the simple good reduction instead of the critically one.

In this paper we are concerned about the relations between these two
notions of good reduction for an endomorphism of $\SR$.
Already in [SzT] the authors remarked that the two notions are not
equivalent and they give examples where none of the two conditions implies the other.

Nevertheless they proved
the following proposition, that for the ease of readers we quote
below, in a slightly simpler form:

\begin{Prop}\label{ST}
Let $K$ be a number field with ring of integers $O_K$, $v$ a finite
place of $K$ and let $\Phi(x)=f(x)/g(x)$ be a rational function
of degree $d$ with coefficients in $O_K$, considered as a rational
function from $\SR$ in itself. Suppose that $R_{\Phi}$ has $2d-2$
elements and that the leading coefficients of $f$, $g$ and
$f'(x)g(x)-f(x)g'(x)$ are all $v$-adic units. Then, if $\Phi$ has
C.G.R. at $v$, it also has S.G.R. at $v$.
\end{Prop}
We have obtained a significant improvement of this result.
\begin{Teo}\label{CPT}
Let $\Phi:\SR\to\SR$ be a morphism of degree $\geq 2$ defined over $K$.
Let $v$ be a finite place of $K$. Let $\Phi_v$ be the map defined as before. Let us suppose that  $\Phi_v$ is separable. Then  the following are equivalent:
\begin{itemize}
\item[a)]   $\Phi$ is C.G.R. at $v$;
\item[b)]   $\Phi$ is S.G.R. at $v$ and $\#\Phi(\mathcal{R}_{\Phi})=\#(\Phi(\mathcal{R}_{\Phi}))_v$.
\end{itemize}
\end{Teo}
We recall that $\Phi_v$ is separable if and only if it is not a $p$-th power as rational function, where $p$ is the characteristic of the residue field $k(v)$ of $v$.

The proposition of Szpiro and Tucker  follows from the above theorem since the fact that the leading term of $f'(x)g(x)-f(x)g'(x)$ is a unit implies in particular that $\Phi_v$ is separable.


If we remove the condition of separability the theorem may not be
true; let us consider for example the polynomial
$\Phi(x)=-3x^4+4x^3$: it has C.G.R. at the prime $3$ but it has
not S.G.R. at $3$, since its reduction $\Phi_3(x)=x^3$ has degree
less than $4$ (note that $\Phi_3$ is not constant). Nevertheless
there are examples where the map is both C.G.R. and S.G.R. at a
prime $p$ but the separability condition does not hold, like the
family of maps $\Phi(x)=x^p$, where $p$ is a prime number.\\
Nevertheless the condition on separability seems to be a good condition. Indeed if $p$, the integral prime under $v$, is bigger than the degree of a map $\Phi$, the $\Phi_v$ is separable if and only if it is non constant. Therefore a direct consequence of Theorem \ref{CPT} is the following corollary:

\begin{Cor}
Let $\Phi:\SR\to\SR$ be a morphism of degree $\geq 2$ defined over $K$.
Let $v$ be a finite place of $K$. Let $p$ be the prime of $\Z$ under the place $v$ and suppose that
$p>\deg(\Phi)$ and $\Phi$ has C.G.R. at $v$. Then $\Phi$
has S.G.R. at $v$ if and only if $\Phi_v$ is not constant.
\end{Cor}

The result of Theorem \ref{CPT} establishes some sufficient conditions to have simple good reduction for covers.
A very general result in this direction is \cite[Thm 3.3]{Fulton}, where Fulton proves and extends some results stated by
Grothendieck in \cite{Grothendieck}. Analogue of Fulton's result for cover of curves, using different methods, are proved by Beckmann in  \cite[Prop. 5.3]{Beckmann}. For similar theorem on plane curves see also \cite{Uz}.
Zannier also wrote another result which is more related to ours. He proved a theorem concerning the good reduction for some particular covers $\SR \to \SR$.
The notion of good reduction used by Zannier is the following one:
using the above notation, a rational map $\Phi$ of $\SR$ defined over a field $L$ has \emph{good reduction}
at a prime $v$ if there exist $a,b\in L$ such that the composed map $\Phi (ax+b)$ has S.G.R. at $v$ and $(\Phi(ax+b))_v$ is separable.
We say that $\Phi$ has \emph{potential good reduction} if it has good reduction over a finite extension of $L$.
Now we are ready to state the Zannier's result:

\bigskip
\noindent {\bf Theorem 1 in \cite{Uz3}} \emph{Let $L$ be a field of characteristic zero, with discrete valuation $v$ having residue field $L_0$ of characteristic $p>0$. Let $\Phi\doteqdot f/g\in L(t)$ be a Belyi cover (i.e. unramified outside $\{0,1,\infty\}$) with $f(t)=\prod_{i=1}^h(t-\xi_i)^{\mu_i}, g(t)=\prod_{j=1}^k(t-\eta_j)^{\nu_j}$ polynomials of positive degree $n$ and the $\xi_i,\eta_j\in\bar{\Q}$ are pairwise distinct. If $\Phi$ does not have potential good reduction at $v$, then $p$ divides the order of the monodromy group and also some non zero integer of the form $\sum_{i\in A}\mu_i-\sum_{j\in B}v_j$ where $A\subset \{1,\ldots, h\}, B\subset \{1,\ldots, k\}$.}

\bigskip
The part of this theorem concerning the divisibility of order of monodromy group could be seen as an application of Beckmann's result
in \cite{Beckmann}, for curves of genus 0. But the method used by Zannier is completely different from the Beckmann and Fulton's ones.
Furthermore Zannier's result gives some new sufficient conditions for the good reduction for Belyi covers. \\
There is a substantial difference between our result and the Beckmann and Zannier's ones. Our Theorem \ref{CPT}
concerns the \textquotedblleft good reduction\textquotedblright \ for a fixed model of a cover $\SR \to \SR$.
The results obtained by Zannier and Beckmann give some sufficient condition for the existence of a model, of a given cover,
with good reduction. For example the polynomial $\Phi(z)=a^2z^2$ for all $a\in \Z$ does not have S.G.R. at all prime dividing
the integer $a$, but it has good reduction in the Beckmann and Zannier's definitions.

Zannier considers only covers $\SR \to \SR$ unramified outside $\{0,\infty, 1\}$, because this covers are strictly related
to the problem of existence of distinct monic polynomials $F,G$ having roots of prescribed multiplicities and $\deg(F-G$)
as small as predicted by Mason's $abc$ theorem. Zannier studied this existence problem in \cite{Uz2} in characteristic 0 and
in \cite{Uz3} in positive characteristic.

 We conclude with an arithmetical and dynamical application of our result.

Let $E$ be an elliptic curve defined over a number field $K$.
Let us consider a fixed model for $E$ given by an equation $y^2=F(x)=x^3+px+q$.
Let $S$ be the minimal finite set of places of $K$ containing all the archimedean ones, all the finite places
above 2 and such that the model is defined over $O_S$ with good reduction outside all finite places outside $S$.
As proved in \cite{SzT} the corresponding Latt\'es map $\Phi(x)=\frac{(F'(x))^2-8xF(x)}{4F(x)}$ has both C.G.R. and S.G.R. at $v$ for all places $v\notin S$. If $P\in E$ then $\Phi(x)$ is the $x$-coordinate of $2P$, where $x$ is the $x$-coordinate of $P$. The set of $K$--rational pre-periodic points of $\Phi$ is the set of $x$--coordinates of the
$K$--rational torsion points of $E$ (see \cite[p.33]{Silverman}). Therefore information about pre-periodic points  for $\Phi$ give us information on torsion points of $E$. This is one of the motivation to study the arithmetic of dynamical systems, and in particular the set of pre-periodic points of rational maps with S.G.R. outside a prescribed set. The application that we shall present involves a theorem proved by Canci in \cite{Canci} which is an extension to pre-periodic points of a result about periodic points due to Morton and Silverman (see \cite{MoS}) in terms of simply good reduction.

Now it is natural to study pre-periodic points of arithmetical dynamical systems, given by maps having $C.G.R.$ outside a prescribed set.
Unfortunately, the notion of C.G.R. may not be preserved under iteration. See for example $\Phi(x)=(x-1)^2$,
where $\Phi$ has C.G.R. everywhere but $\Phi^2$ (i.e. $\Phi\circ\Phi$) does not have C.G.R at 2. On the contrary
the condition of  S.G.R. is preserved under iteration. So it is a good notion for dynamical studies.

Before stating the dynamical result obtained by using our Theorem \ref{CPT}, Theorem 1 in \cite{Canci} and Corollary B in \cite{MoS}, we set some notation:
given a point $P$ in $\SR$ we denote by $O_{\Phi}(P)$ the orbit
of $P$ under the map $\Phi$, that is the set
$\{\Phi^n(P)\;|\;n\in\N\}$, where
$\Phi^n$ is the $n$-th iterate of
$\Phi$. Let $K$, $S$, $v$ and $\Phi_v$ be as above. Let $\#{\rm PrePer}(\Phi, \SR(K))$
be the cardinality of the set of $K$--rational pre-periodic points of the map $\Phi$.

\begin{Cor}\label{UBCforCGR}
Let $d$, $D$ and $t$ be fixed integers with $d\geq 2$. Then there exists a
constant $C=C(d,D,t)$ such that given a number field $K$ of degree
$D$, a finite set of places $S$ of $K$ of cardinality less than $t$,
a rational map $\Phi:\SR\to\SR$ of degree $d$ defined over $K$,
such that $\Phi$ has C.G.R. at every place $v$ outside $S$ and
$\Phi_v$ is not constant for each $v$ not in $S$,
then the following inequalities holds:
$$\#{\rm PrePer}(\Phi, \SR(K))\leq C(d,D,t).$$
\end{Cor}


%

Corollary \ref{UBCforCGR} represents a very particular case of Uniform Boundedness Conjecture for pre-periodic points stated by Morton and Silverman in \cite{MoS}.

It is worth noticing that computationally speaking, given a place $v$
of $K$, it is easier to check that $\Phi_v$ is not constant than
checking that $\Phi$ has S.G.R. at $v$. In the first case we have
to compute $\binom{d+1}{2}$ determinants of $2\times2$ matrices,
while in the second case we have to compute a determinant of a
$(2d+2)\times(2d+2)$ matrix.
In the first case we have to do an $O(d^2)$ number of calculations and in the second case the number is an $O(d^3)$.
Note that the LU decomposition of a matrix reduce the number of operations from $O(d!)$, necessary by using the  Leibniz rule,
to $O(d^3)$ calculations (e.g. see \cite{QSS}).

\subsection*{Acknowledgments}

We worked on this problem during our stay in the Hausdorff Research Institute for Mathematics in Bonn for the trimester program
on Algebra and Number Theory. We wish to thank the director Matthias Kreck and all the staff of the Institute for the given support while using the
facilities of the Institute.

The notion of Critically Good Reduction in Szpiro and Tucker's article has been pointed out to us by Pietro Corvaja: we wish to thank him for that. We thank Thomas Tucker for reading the article and pointing us out some inaccuracies.

Finally we thank Umberto Zannier for giving us precious references related to our work.

\section{Proof of main results}

From now on $K$ will be a number field, $v$ a non-archimedean valuation of $K$ and $O_v$ the associated valuation ring.
For any polynomial $h(x)\in O_v[x]$, $h_v(x)$ will denote the polynomial obtained from $h$ by reduction of its coefficients modulo $v$. Also for any $\alpha\in K$, we will denote its reduction modulo $v$ with the symbol $\alpha_v$. \\
For a given endomorphism $\Phi$ of $\SR$ with $\Phi([X:Y])=[F(X,Y):G(X,Y)]$ where $F,G\in O[X,Y]$ are homogeneous polynomials of the same degree $d$ without common factors, we associate the rational function $\Phi(x)=f(x)/g(x)$, that with abuse of notation we denote again by $\Phi$, where $f(X/Y)=F(X,Y)/Y^d$ and $g(X/Y)=G(X,Y)/Y^d$. We can reverse this argument, so to each rational function $\Phi\in K(x)$ we associate a unique endomorphism $\Phi$ of $\SR$.
From now on we suppose that $\Phi(x)=f(x)/g(x)$ is a rational function defined over $K$ written in $v$-reduced form.

Roughly speaking if
$$\Phi(x)=\frac{f(x)}{g(x)}=\frac{a_dx^d+\dots+a_0}{b_dx^d+\dots+b_0}$$
is represented in $v$-reduced form, that means $a_i,b_j\in O_v$ for $0\leq i\leq d$ and $0\leq j\leq d$, $a_d\neq $ or $b_d\neq 0$ and
$$\min \{v(a_d),v(a_{d-1}),\ldots,v(a_0),v(b_d),\ldots,v(b_0)\}=0,$$
then
$$\Phi_v(x)=\frac{f_v(x)}{g_v(x)}=\frac{(a_d)_vx^d+\dots+(a_0)_v}{(b_d)_vx^d+\dots+(b_0)_v}.$$
Note that in general $f_v$ and $g_v$ may not be coprime.\\
We define the following polynomial in $O_v[x]$
\begin{equation}\label{diff}
\Phi^{(1)}(x)\doteqdot f'(x)g(x)-f(x)g'(x)
\end{equation}
Its degree is less or equal to $2d-2$. It is quite easy to check that
$$
\mathcal{R}_{\Phi}\setminus\{\infty\}=\{x\in\overline \Q\;|\;\Phi^{(1)}(x)=0\}.
$$
and $\infty$ is a ramification point if the polynomial has degree $< 2d-2$. 

\bigskip

It may happen that the set of primes of critically bad reduction
increase if we compose with homotheties which are not
$v$-invertible, like for example: $f(x)=x^2+x$ and $A(x)=x/3$. The
map $f$ has C.G.R. in $3$ but the map $f^A=A\circ f\circ A^{-1}$ has not.
The following lemma shows that the two notions of good reduction at
a place $v$ are preserved under equivalence with $v$-invertible
elements of ${\rm PGL}(2,O_v)$.

\begin{Lemma}\label{GR}
Suppose that $\Phi$ has S.G.R. (C.G.R., respectively) at a place
$v$. Suppose that $\alpha$, $\beta$ are invertible rational maps associated to the elements $A,B\in{\rm PGL}(2,O_v)$ respectively. Then
$\alpha\circ\Phi\circ \beta$ has S.G.R. (C.G.R., respectively) at $v$.
\end{Lemma}

\Dim For the S.G.R. we use the fact that the composition of maps with S.G.R. has S.G.R. (see  \cite[Thm2.18]{Silverman}).
For the C.G.R. we use \cite[Prop. 2.9]{Silverman}: given $P_1,P_2\in\SR$ such that $P_1\not\equiv P_2 \pmod{v}$ then if $A\in{\rm PGL}(2,O_v)$
we have that $A(P_1)\not\equiv A(P_2) \pmod{v}$. $\Box$
\bigskip

The condition of the above lemma is not necessary, consider for example:
$f(x)=x^2+3x$ and $A(x)=x/3$, then $f$ as well $f^{A}=A\circ f\circ
A^{-1}$ have C.G.R. at $3$ even if  $A\not\in {\rm PGL}(2,\Z_{(3)})$.

We shall use the following equivalence relation:

\begin{Def} Two rational maps $\Phi$ and $\Psi$ are $v$--equivalent if there exist two rational maps $\alpha$ and $\beta$ associated to two invertible elements $A,B\in{\rm PGL}(2,O_v)$ respectively such that $\Phi=\alpha\circ\Psi\circ\beta$.
\end{Def}

In general, the reduction modulo $v$ of rational maps does not commute with the composition of rational functions. For example consider
$$\Phi(x)=\dfrac{x^2+x}{x+p}\quad \text{and}\quad \Psi(x)=px$$ for a given prime integer $p$. We have
$$(\Phi\circ\Psi)_p=\dfrac{x}{x+1}\quad \text{and}\quad\Phi_p\circ\Psi_p=1.$$
But if the maps $\Phi$ and $\Psi$ have both S.G.R. at $v$ then
$$(\Phi\circ\Psi)_v=\Phi_v\circ\Psi_v.$$
(See Theorem 2.18 in \cite{Silverman}).
In order to have the commutativity of reduction modulo $v$ and composition it is not always necessary to have the S.G.R. for both maps. For example the following result holds:

\begin{Lemma}\label{commut}
Let $\Phi$ be an endomorphism of $\SR$ defined over $K$. Let $\alpha$ and $\beta$ two $v$-invertible rational maps (i.e. they are associated to two elements in PGL$(2,O_v)$). Then it holds
$$(\alpha\circ\Phi\circ\beta)_v=\alpha_v\circ\Phi_v\circ\beta_v.$$
\end{Lemma}
\begin{proof}
Let
$$\alpha(x)=\dfrac{ax+b}{cx+d}\ \ ,\ \ \Phi(x)=\frac{f(x)}{g(x)}$$
be represented in $v$-reduced form. Now observe that the function
\begin{equation}\label{reducedf}(\alpha\circ \Phi)(x)=\dfrac{af(x)+bg(x)}{cf(x)+dg(x)}\end{equation}
is represented in $v$-reduced form. This follows by considering a representation of $\alpha^{-1}$ in $v$-reduced form:\\
let
$$\alpha^{-1}(x)=\dfrac{lx+r}{sx+t},$$ where $l,r,s,t\in O_v$ and the following identity holds
$$\begin{pmatrix}l& r\\s&t\end{pmatrix}\begin{pmatrix}a& b\\c&d\end{pmatrix}=\begin{pmatrix}1& 0\\0&1\end{pmatrix},$$
then
$$\begin{pmatrix}l& r\\s&t\end{pmatrix}\begin{pmatrix}af(x)+bg(x)\\cf(x)+dg(x)\end{pmatrix}=\begin{pmatrix}f(x)\\g(x)\end{pmatrix}.$$
If (\ref{reducedf}) were not in the reduced form, then also $\Phi=f/g$ would not too.

Therefore
$$
(\alpha\circ \Phi)_v(x)=\dfrac{a_vf_v(x)+b_vg_v(x)}{c_vf_v(x)+d_vg_v(x)}=(\alpha_v\circ\Phi_v)(x).
$$
A similar argument works to prove that for any rational function $\Psi$ defined over $K$, then $(\Psi\circ\beta)_v=\Psi_v\circ\beta_v$. Now we consider $\Psi =\alpha\circ\Phi$ and we obtain
$$\alpha_v\circ\Phi_v\circ\beta_v=(\alpha\circ\Phi)_v\circ\beta_v=(\alpha\circ\Phi\circ\beta)_v.$$
\end{proof}

\begin{Oss}\label{ext}Given a number field $K$ and a place $v$ of $K$, the notion of critically good reduction
is independent of the extension of $v$ to $\overline{\Q}$. Furthermore it is clear by definition of C.G.R. that,
fixed an arbitrary finite extension $L$ of $K$, a rational map $\Phi$ defined over a number field $K$ has C.G.R. at $v$
if and only if $\Phi$, as a rational map defined over $L$, has C.G.R. at $v$. In this way, without loss of generality,
up to enlarging $K$, we can suppose that all ramification points of $\Phi$ are $K$--rational. The same holds also for S.G.R.
in the sense that it is completely trivial that the notion of simply good reduction is stable by extension to a finite extension $L$ of $K$.
Since the action of ${\rm PGL}(2,O_v)$ is transitive on the pairs of elements of $\SR(K)$ that does not have the same reduction
modulo $v$, we can suppose $K$ enlarged so that if a rational map $\Phi$ has C.G.R. at a place $v$, then we may assume
 that $\{0,\infty\}\subset \mathcal{R}_{\Phi}$ and
$\Phi(0)=0$ and $\Phi(\infty)=\infty$. It is sufficient to take instead of $\Phi$ the composition $\alpha\circ\Phi\circ\beta$ where $\alpha,\beta$ are suitable $v$--invertible rational maps associated to some elements of ${\rm PGL}(2,O_v)$.
Note that by Lemma \ref{commut} it follows immediately that $\Phi_v$ is separable if and only if $(\alpha\circ\Phi\circ\beta)_v$ is separable.

\end{Oss}

Now we state a simple Lemma that contains some characterizations of being S.G.R. at $v$.
\begin{Lemma}\label{lem:SGR}
For a morphism $\Phi:\SR\to\SR$ of degree $\geq 1$ the following are equivalent:
\begin{itemize}
\item[a)] $\Phi$ has S.G.R. at a finite place $v$;
\item[b)] $\Phi_v$ is not constant and for any $x_1,x_2\in \SR$ if  $x_1\equiv x_2\mod v$ then  $\Phi(x_1)\equiv \Phi(x_2)\mod v$;
\item[c)] $\Phi_v$ is not constant and there exist $w,z\in \SR$ with $w\not \equiv z\mod v$ such that for any $x_1,x_2\in \SR$ with $\Phi(x_1)=w$ and $\Phi(x_2)=z$ then $x_1\not\equiv x_2\mod v$.
\end{itemize}
\end{Lemma}
\begin{proof}
$a)\Rightarrow b)$.  If one considers the following commutative diagram
\begin{center}
$\begin{diagram}
\node{\ \mathbb{P}^1_{k(v)}}\arrow{e,t}{\Phi} \arrow{s,l}{}
\node{\ \mathbb{P}^1_{k(v)}}\arrow{s,r}{}\\
\node{\ \mathbb{P}^{1}_{O_v}}\arrow{e,t}{{\Phi}_{O_v}}\node{\ \mathbb{P}^{1}_{O_v}}
\end{diagram}$\end{center}
where the vertical map are the natural closed immersions,
then it is easy to prove the above assertion.

$b)\Rightarrow c)$. This is immediate.

$c)\Rightarrow a)$. Let $\Phi(x)=f(x)/g(x)$ be a rational function defined over a
number field $K$, with $f,g\in O_v[x]$ coprime, written in $v$-reduced form. By
lemma \ref{GR} and  Remark \ref{ext}, up to enlarging $K$ and taking a suitable element in the equivalence class of $\Phi$,  we can assume that $w=0$, $z=\infty$, and $\Phi(\infty)=\infty$. So $\deg f>\deg g$.
We observe that, by hypothesis, any preimage of $0$  has non negative valuation, which means that modulo $v$ does not coincide with $\infty$. This and the fact that $\phi_v$ is not constant imply that $f(x)$  has $v$-invertible leading coefficient.
In this situation we have that $\Phi$ has S.G.R. at $v$ if and only if
$f$ and $g$ have no common factors modulo $v$.
But this would contradict the statement in c).
\end{proof}
The above characterizations of S.G.R. (especially part c)  will be used just to shorten some of the following proofs.\\
On the contrary the next lemma, which gives another characterization of C.G.R., will play an important role in the proof of Theorem \ref{CPT}.

\begin{Lemma}\label{lcphi1}
A morphism $\Phi:\SR\to\SR$ of degree $\geq 2$ has C.G.R. at a finite place $v$ and $\Phi_v$ is separable if and only if
\begin{itemize}
\item[1)]  $\Phi_v$ is not constant;
\item[2)] if $x_1\in \mathcal{R}_{{\Phi}}$, $x_2\in \Phi^{-1}(\Phi(\mathcal{R}_{\Phi}))$ then $x_1\equiv x_2 \mod v$ if and only if $x_1=x_2$;
\item[3)] the ramification index of any ramification point is not divisible by the characteristic of the residue field $k(v)$.
\end{itemize}
\end{Lemma}
\begin{proof}
Let us suppose that $\Phi$ has C.G.R. at $v$. Let $x_1$ and $x_2$ be as in 2).
By Remark \ref{ext}, without loss of generality  we may assume  $x_1=\infty, \Phi(\infty)=\infty$ and $x_2\in \Phi^{-1}(0)\bigcup \Phi^{-1}(\infty)$. In particular $\deg(f)>\deg(g)$ and we are assuming that $0\in\Phi(\mathcal{R}_\Phi)$.
Since C.G.R. is stable under a finite extension of the field of the coefficients of $\Phi$,
we may assume that all the polynomials we are dealing with have linear factors over $K$.
So let us write down the factorization of $\Phi^{(1)}(x)$ (see (\ref{diff}))
\begin{equation}\label{**}
\Phi^{(1)}(x)=\theta\prod_{k}(x-\alpha_k)^{e_k}
\end{equation}
where $\theta\in O_v$ is the leading coefficient of $\Phi^{(1)}(x)$
and $\{\alpha_k\}_k=\mathcal{R}_{\Phi}\setminus\{\infty\}$. Note that, since $\Phi$ has C.G.R. at $v$ then
$\alpha_k$ is in $O_v$ for all $k$. Since $\deg(f)>\deg(g)$ by direct computation  we get that
\begin{equation}\label{theta}
\theta={\rm lc}(f){\rm lc}(g)(\deg f-\deg g)
\end{equation}
Therefore we get that
$$
\Phi^{(1)}\not\equiv0\pmod v\Leftrightarrow \theta\not\equiv0\pmod v
$$
since each $\alpha_k$ is a $v$-integer.

Note that if $f_v$ and $g_v$, the reduction modulo $v$ of polynomials $f$ and $g$, are such that
$$f_v(x)=h(x)f_1(x) \ \ ,\ \ g_v(x)=h(x)g_1(x)$$
with suitable $h,f_1,g_1\in k(v)[x]$ and $f_1,g_1$ coprime, then $h(x)$ is not zero because $\Phi=f/g$ is in $v$-reduced form. Furthermore $(\Phi^{(1)})_v(x)$ the reduction modulo $v$ of the polynomial $\Phi^{(1)}(x)$ is such that
$$(\Phi^{(1)})_v(x)=h(x)^2(f_1^\prime g_1-f_1g_1^\prime).$$
Hence $(\Phi^{(1)})_v$ is zero if and only if $f_1^\prime g_1-f_1g_1^\prime$ is zero, which is equivalent to $\Phi_v=f_1/g_1$ inseparable.

Therefore $\Phi_v$ is separable if and only if the leading coefficients of $f$ and $g$ are $v$-units and $\deg f-\deg  g$ is not divisible by the characteristic of the residue field $k(v)$. But $\deg f- \deg g$ is the ramification index of $\infty$, therefore this means  the ramification index of $\infty$ is not divisible by the characteristic of $k(v)$.  \\
Since the leading coefficients of $f(x)$ and $g(x)$  are
 $v$-units then all the elements in $\Phi^{-1}(0)\cup \Phi^{-1}(\infty)$ different from $\infty$ are not equivalent to $\infty$ modulo $v$. In particular this holds for $x_2$. We have so proved that under the condition that $\Phi$ has C.G.R. at $v$ this implies that $\Phi_v$ separable is equivalent to 1), 2) and 3).

Now we prove that conditions 1), 2) and 3) imply that $\Phi$ has C.G.R. at $v$. By 2) to prove that $\Phi$ is C.G.R. it is sufficient to verify the condition on the branch locus. We have to prove that for any pair of distinct points $y_1,y_2\in \Phi(\mathcal{R}_\Phi)$ then $y_1$ and $y_2$ are also distinct modulo $v$. Again from Remark \ref{ext}, and by condition 2), we can suppose that $y_1=\infty$, $\Phi(\infty)=\infty$, $\Phi(0)=y_2$, and $ 0,\infty\in\mathcal{R}_\Phi$. We represent $\Phi$ in the following $v$-normal form
\begin{equation}\label{normform}
\Phi(x)=\frac{a_dx^d+\dots+a_0}{b_mx^m+\dots+ b_0}=\dfrac{a_d\prod_i(x-\eta_i)}{b_m\prod_j(x-\rho_j)}
\end{equation}
where  $d>m+1$ and $a_i, b_j\in O_v$ for all indexes $i,j$. We suppose $K$ enlarged so that it contains all roots $\eta_i$ and $\rho_j$. Note that $y_2=a_0/b_0$. Since any root $\rho_j$ of the denominator ${b_mx^m+\dots+ b_0}$ is in the fiber of $\infty\in \mathcal{R}_\Phi$ and also $0$ is a ramification point, then by 2) each $\rho_j$ has to be a $v$-unit. Since $\Phi_v$ is not constant, then $b_0=b_m\prod_j\rho_j$ is a $v$-unit, thus $v(y_2)\geq 0$. Therefore the reduction modulo $v$ of $y_2$ is not $\infty$. This proves that $\Phi$ has C.G.R at $v$.\\
\end{proof}

\begin{proof}[Proof of Theorem \ref{CPT}]
We first prove that $a)\Rightarrow b)$.
\\
Let $\Phi(x)=f(x)/g(x)$ be a rational function defined over $K$, with $f,g\in O_v[x]$ coprime, written in $v$-reduced form. By Remark \ref{ext} we can assume that $\{0,\infty\}\subset \mathcal{R}_{\Phi}$ and
$\Phi(0)=0$, $\Phi(\infty)=\infty$. In
particular we also have that $\deg(f)>\deg(g)$. 
Let us use the notation as in the proof of Lemma \ref{lcphi1} in particular the notation in (\ref{**}).
Furthermore we suppose $K$ enlarged so that it contains all roots of the polynomials $f,g$ and $\Phi^{(1)}$.

Let us suppose that $\Phi$ has not S.G.R. at $v$. This means that there exist $\beta_1\in \Phi^{-1}(0)$ and $\beta_2\in \Phi^{-1}(\infty)$ such that $\beta_1\equiv \beta_2\mod v$.
Let us define $\beta_v\doteqdot(\beta_1)_v=(\beta_2)_v$. Note that it is not possible that $\beta_v$ is $\infty$, by part 2) of Lemma \ref{lcphi1}. Let $(\Phi^{(1)})_v$ be the polynomial obtained from $\Phi^{(1)}$ by reduction of its coefficients modulo $v$. Since
$$
(\Phi^{(1)})_v(x)=f'_v(x)g_v(x)-f_v(x)g'_v(x),
$$
we have that  $\beta_v$ is a root of the polynomial  $(\Phi^{(1)})_v$. Since $\Phi_v$ is separable, the polynomial $(\Phi^{(1)})_v$ is not zero. Thus any root of the polynomial $(\Phi^{(1)})_v$ is the reduction modulo $v$ of a ramification point $\alpha_i$ for $\Phi$. Therefore $\beta_v$ is equal to the reduction modulo $v$ for one $\alpha_i$ for a suitable index $i$.
 Clearly $\alpha_i\neq \beta_1$ or $\alpha_i\neq \beta_2$. This contradicts 2) of Lemma \ref{lcphi1}.

We now prove $b)$ implies $a)$.
Since $K$ has characteristic 0, then the Riemann-Hurwitz Formula in our situation becomes:

\begin{equation}\label{RieHur} 2\deg\Phi -2=\sum_{P\in\SR(\overline{\Q})}(e_P(\Phi)-1).\end{equation}

Since the map $\Phi_v$ is separable, the Riemann-Hurwitz Formula holds for $\Phi_v$. But the map is defined over $k(v)$ and the characteristic is positive, hence we could have wild ramification. Let R$_{\Phi_v}$ be the ramification divisor associated to the map $\Phi_v$. By  \cite[Prop. 2.2]{Hartshorne} we have that

$$ \deg {\rm R}_{\Phi_v}\geq \sum_{P\in\SR(\overline{k(v)})}(e_P(\Phi_v)-1).$$
Since $\Phi$ has S.G.R. at $v$, by Riemann-Hurwitz Formula we have

\begin{equation}\label{RamDeg}2\deg\Phi-2=2\deg\Phi_v-2=\deg {\rm R}_{\Phi_v}.
\end{equation}
Moreover for any ramification point $P$ of $\Phi$, the point $P_v\in \SR(\overline{k(v)})$ (i.e. the reduction mod $v$ of the point $P$) is a ramification point for $\Phi_v$ and the ramification index $e_{P_v}(\Phi_v)$ is equal or grater than the ramification index $e_{P}(\Phi)$. Furthermore, by the condition of non singularity of the branch locus $\Phi(\mathcal{R}_\Phi)$ of $\Phi$, if $Q_1,Q_2\in \Phi(\mathcal{R}_\Phi)$ are distinct points, then by Lemma \ref{lem:SGR} the sets $(\Phi^{-1}(Q_1))_v$ and $(\Phi^{-1}(Q_2))_v$ are disjoint. Thus the following inequalities hold
$$\deg {\rm R}_{\Phi_v}\geq \sum_{P\in\SR(\overline{k(v)})}(e_P(\Phi_v)-1)\geq \sum_{P\in\SR(\overline{\Q})}(e_P(\Phi)-1).$$
Now suppose that there exist two distinct ramification points $P_1,P_2\in \mathcal{R}_\Phi$ such that $(P_1)_v=(P_2)_v$, then by the S.G.R. condition we have $(\Phi(P_1))_v=(\Phi(P_2))_v$ and by the condition on the branch locus the identity $\Phi(P_1)=\Phi(P_2)$ holds. In this situation the second inequality becomes strict:
$$\deg {\rm R}_{\Phi_v}\geq \sum_{P\in\SR(\overline{k(v)})}(e_P(\Phi_v)-1)> \sum_{P\in\SR(\overline{\Q})}(e_P(\Phi)-1)$$
which gives a contradiction by identities (\ref{RieHur}) and (\ref{RamDeg}).
\end{proof}
\bigskip


\section{Some examples}

In this section we will also consider some cases in which the residue map is not separable, so that one can not apply directly the previous theorem.
The following example gives an example
in which the implication $b)\Rightarrow a)$ in  Theorem $\ref{CPT}$ does not hold without the condition about separability. It also shows that the condition C.G.R. is not stable under composition of maps.
\begin{Esempio}\label{E1}
The set of ramification points of the rational map $\Phi(x)=(x-1)^2$ is $\mathcal{R}_\Phi=\{\infty, 1\}$ and the branch locus is $\Phi(\mathcal{R}_\Phi)=\{\infty, 0\}$.  The set of ramification points of $\Phi^2=\Phi\circ\Phi$ is $\mathcal{R}_{\Phi^2}=\{\infty, 0,1,2\}$ and the branch locus is $\Phi^2(\mathcal{R}_{\Phi^2})=\{\infty, 0, 1\}$. Therefore $\Phi$ has C.G.R. at all finite places $v$ and $\Phi^2$ does not have C.G.R. at $2$. Furthermore, note that $\Phi^2$ has S.G.R. at any finite places $v$, the branch locus of $\Phi^2$ is not singular modulo any   finite places $v$. Therefore Theorem \ref{CPT} does not apply to $\Phi^2$ because it is inseparable modulo 2.
\end{Esempio}
\begin{Prop}\label{oss:caso particolare}
Let $\Phi:\SR\rightarrow \SR$ be a morphism defined over $K$ of degree $\geq 2$ such that a point $x$ belongs to the ramification locus if and only if this holds for any point of the fiber $\Phi^{-1}(\Phi(x))$ (e.g. a Galois cover). Then the following are equivalent:
\begin{itemize}
\item[a)]$\Phi$ is S.G.R and C.G.R. at $v$;
\item[b)]$\Phi_v$ not constant and  $\#\mathcal{R}_{\Phi}=\#(\mathcal{R}_{\Phi})_v$.
\end{itemize}
\end{Prop}
\begin{proof}
Clearly one has only to prove that $b)$ implies $a)$.
We first prove that the branch locus is not singular.
We have just to prove that if $x_1,x_2\in \mathcal{R}_{\Phi}$ and $\Phi(x_1)\neq \Phi(x_2)$ then $\Phi(x_1)\not\equiv\Phi(x_2)\mod v$. Since the ramification locus is not singular then $x_1\not \equiv x_2\mod v$, so we can suppose that $x_1=0$, $x_2=\infty$ and $\Phi(\infty)=\infty$. Therefore we can represent the map $\Phi$ in the following $v$-normal form
$$
\Phi(x)=\frac{f(x)}{g(x)}=\frac{a_dx^d+\dots+a_0}{b_mx^m+\dots+ b_0}
$$
where $a_i, b_i\in O_v$ and $d>m+1$. We have just to prove that $\Phi(0)\not \equiv \infty \mod v$, i.e. $v(a_0)\geq v(b_0)$. But in fact  $b_0$ is a $v$-unit. Indeed, since $\Phi_v$ is not constant then $g\not\equiv 0 \mod v$. So if $b_0\equiv 0\mod v$, since $\Phi(0)\neq\infty$ by our assumption, there would be a non zero point $z\in \Phi^{-1}(\infty)$ such that $z\equiv 0\mod v$. But this would contradict the hypothesis since any point of $\Phi^{-1}(\infty)$ is ramified and the ramification locus is not singular.

Now it is clear that $\Phi$ has S.G.R. at $v$ using the hypothesis and Lemma  \ref{lem:SGR}.
\end{proof}
%


In the case the degree of the map is equal to 2, we have the following simple situation:

\begin{Prop} Let $K$ be a number field, $v$ a finite places of $K$ and $\Phi$ a rational map of degree 2. Then:
\begin{enumerate}
\item If $v$ does not lie above $2$, then
$$\Phi\ \text{S.G.R. at $v$}\ \Leftrightarrow \Phi \ \text{C.G.R. at $v$ and $\Phi_v$ is not constant}.$$
\item If $v$ lies above $2$, then the following are equivalent.
\begin{itemize}
\item[i)]$\Phi$ is S.G.R. at $v$ and $\Phi_2$  factors through the relative Frobenius of $\mathbb{P}^1_{O_v/2 O_v}$; 
    \item[ii)] $\Phi$ is C.G.R. at $v$ and $\Phi_v$ is not constant.
    \end{itemize}
\end{enumerate}
\end{Prop}
\begin{Oss}
In the above statement, with $\Phi_2$ we mean the restriction of $\Phi$ to the scheme ${\SR}_{O_v/2O_v}$. For the notion of the relative Frobenius we refer to \cite[sec. 3.2.4]{Liu}.
\end{Oss}
\begin{proof} In both cases the \textit{if} part, except the sentence on the inseparability of $\Phi_2$, follows from  Proposition \ref{oss:caso particolare}. Now if $\Phi$ is C.G.R by the argument in Remark \ref{ext}, we can assume that $\Phi$ is of the form $ax^2$. Then, if $v$ is above $2$, $\Phi_2$ is purely inseparable.

Now let us suppose that $\Phi$ is S.G.R. at $v$. By the argument in Remark \ref{ext} we can suppose that $\infty\in\mathcal{R}_{\Phi}$ and $\Phi(\infty)=\infty$. Therefore $\Phi$ has the form $ax^2+bx+c$ with $a,b,c\in K$. By the above Proposition   \ref{oss:caso particolare} we have only to check that the ramification locus is not singular at $v$.
Since $\Phi$ has S.G.R. at $v$, then $a$ is a $v$--unity and $b,c$ are $v$--integers. The ramification points  of $\Phi$ are:
$$ \mathcal{R}_\Phi=\left\{\infty, -\frac{b}{2a}\right\}.$$
Now, if $v$ is not above $2$, then $2a$ is a  $v$-unity so that $\Phi$ is C.G.R. at $v$. If $v$ is above 2 and $\Phi_2$ is purely inseparable, then $2\mid b$  are in $O_v$. Hence, also in this case $\Phi$ has C.G.R. at $v$.
\end{proof}

\section{An application to  arithmetical dynamics}
  As already remarked in the introduction,
the notion of C.G.R. does not have a good behavior with dynamical problems
associated to a rational map.

The next example shows that the behavior of the critical good reduction under iteration of a rational map could be truly bad. Indeed, we give an example of rational map $\Phi$ defined over $\Q$, where with simple calculations, we see that it does not exist a finite set $S$ of valuations of $\overline{\Q}$  such that all iterated of $\Phi$ has C.G.R. at all finite valuations outside $S$.

\begin{Esempio} Consider the rational function $\Phi(x)=x(x-1)$. Its set of ramification points is $\mathcal{R}_\Phi=\{\infty,1/2\}$ and its branch locus is  $\Phi(\mathcal{R}_\Phi)=\{\infty,-1/4\}$. Hence the map $\Phi$ does not have C.G.R. only at $2$.  Let $\Phi^n$ be the $n$--th iterated map of $\Phi$. Denote by $\mathcal{B}_n$ the branch locus of $\Phi^n$ that is
$$\mathcal{B}_n=\Phi^n(\mathcal{R}_{\Phi^n})=\bigcup_{i=1}^n\Phi^i(\mathcal{R}_\Phi)=\{\infty\}\cup \{\Phi^i(1/2)\mid 1\leq i\leq n\}$$
Note that the element $1/2$ is not a preperiodic point for $\Phi$. Indeed we have
$$\Phi^i(1/2)=\frac{a_i}{2^{i+1}}\quad \text{for any index $i\geq 1$}$$
where the $a_i's$ are suitable odd integers. Therefore the sequence $\{\mathcal{B}_n\}$ of sets of elements in $\{\infty\}\cup\Q$ is strictly increasing. Let $S$ be a finite fixed set of finite places of $\overline{\Q}$. Let $p$ be the minimum of the prime integers that are below a valuation not in $S$. Thus any set of elements in $\{\infty\}\cup\Q$ of cardinality bigger than $p+1$ is singular modulo at least at one valuation outside $S$. Hence it does not exist a finite set $S$ of valuation of $\overline{\Q}$ such that all iterated of $\Phi$ has C.G.R. at all finite valuations outside $S$.
\end{Esempio}

By Theorem \ref{CPT} if $\Phi$ is a rational map with C.G.R at a valuation $v$ and $\Phi_v$ is separable modulo $v$, then $\Phi$ has S.G.R. at $v$. Therefore it has good behavior in a dynamical sense. The proof of Corollary \ref{UBCforCGR} is a simple application of our Theorem \ref{CPT},  \cite[Corollary B]{MoS} and \cite[Theorem 1]{Canci}.

In \cite[Corollary B]{MoS} Morton and Silverman proved that if $\Phi$ is a rational map of degree $\geq 2$ which has good reduction outside $S$ (a finite fixed set of valuations of $K$ containing all the archimdean ones with $|S|=t$)  and $P\in\SR(K)$ is a periodic point with minimal period $n$, then the inequality
\begin{equation*} n\leq \left[12(t+1)\log(5(t+1))\right]^{4\left[K:\mathbb{Q}\right]}\end{equation*} holds.

In \cite[Theorem 1]{Canci}, Canci extended the Morton and Silverman's result to any finite orbit (so he considered also pre-periodic points). With the same above hypothesis as in \cite[Corollary B]{MoS}, the Canci's result says that there exists a number $c(t)$, depending only on $t$, such that the length of every finite orbit in $\SR(K)$, for rational maps with good reduction outside $S$, is bounded by $c(t)$.
The number $c(t)$ can be chosen equal to
\begin{align*}\left[e^{10^{12}}(t+1)^8(\log(5(t+1)))^8\right]^{t}.\end{align*}
%
%
\begin{proof}[Proof of Corollary \ref{UBCforCGR}]
Let $\Phi$ be an endomorphism of $\SR$ as in the hypothesis of the corollary. For any prime integer $p\leq \deg\Phi$ we consider all valuations $v_p$ over $K$ which extend the valuation associate to $p$. We enlarge $S$ adding all these valuations $v_p$ for all $p\leq \deg \Phi$. The cardinality of the new set $S$ depends only on $t$, on $d$ the degree of the map and $D$ the degree of $K$ over $\Q$. With this enlarged set $S$, for any $v\notin S$, the reduced map $\Phi_v$ is separable if and only if it is not constant. Therefore the map $\Phi$ has S.G.R. at any valuation outside $S$. We denote by $b(t,d,D)$ the lowest integer bigger than the Morton and Silverman's bound, which depends on the cardinality of the enlarged set $S$. There exists a bound $B(t,D,d)$ which bounds the cardinality of the set of $K$--rational periodic points of $\Phi$. Indeed, any $K$--rational point is a fixed point for the map $\Phi^{b(t,d,D)!}$. Hence we could take $B(t,D,d)= b(t,d,D)!+1$. By the Canci's bound $c(t,d,D)$, which depends on the cardinality of the enlarged set $S$, any $K$--rational periodic point of $\Phi$ is contained in at most $d^{c(t,d,D)}$ finite orbits. Thus we could take $C(t,d,D)=B(t,D,d)d^{c(t,d,D)}c(t,d,D)$.
\end{proof}
Any number depending on $d,D,t$ in our proof could be not optimal. Our aim was to show the existence of a bound $C(d,D,t)$ as in the statement of Corollary \ref{UBCforCGR} and not to find an optimal limit.

\end{document}